\documentclass{article}
\usepackage{amssymb}
\newcommand{\R}{{{\mathbb R}}}
\newcommand{\N}{{{\mathbb N}}}
\newcommand{\Z}{{{\mathbb Z}}}
\newcommand{\T}{{{\mathbb T}}}

\newtheorem{lema}{Lemma}[section]
\newtheorem{teor}{Theorem}[section]

\newtheorem{definicion}{Definition}[section]
\newtheorem{nota}{Remark}[section]

\def\qed{\hbox to 0pt{}\hfill$\rlap{$\sqcap$}\sqcup$\medbreak}

\title{Positive Solution of  Singular BVPs for System of Dynamic Equations on Time Scales}
\author{{\sc Ariadna Lago}, {\sc Victoria
Otero--Espinar\footnote{Corresponding author,
E-mail: mvictoria.otero@usc.es}} and {\sc Tania Pernas} \\
{Departamento de
An\'alise Matem\'atica,}\\ {Facultade de Matem\'aticas,}\\
{Universidade de Santiago de Compostela,}\\{Galicia,} $\,$ {\sc
{\sc Spain}}}

\begin{document}
\date{}
\maketitle

\begin{abstract}
This paper is devoted  to derive some necessary and sufficient
conditions for the existence of  positive solutions to a singular second order system of dynamic equations with Dirichlet boundary conditions. The results are obtained by employing the fixed-point theorems and the method of the lower and upper solutions.
\end{abstract}



\section{Introduction}

The main purpose of this paper is to establish existence results for the second order Dirichlet system 
$$
(P)\left\{\begin{array}{l}
x_i^{\Delta\Delta}=f_i(t,x^\sigma(t)),\; t\in (J^\kappa)^o,\\
\noalign{\bigskip}x_i(a)=A_i, \, x_i(\sigma^2(b))=B_i,\; i=1,2,\dots,n.
\end{array}\right.
$$
with $f=(f_1,\dots,f_n)$, $i=1,2,\dots,n$, where $f_i:(J^\kappa)^o\times\mathcal{A}\rightarrow \R$,  $\mathcal{A} \subset \R^n$ and $J$ is a time scale interval. The nonlinearity $ f_i (t, x)$ may be singular at $ x_i$, $ i = 1, \dots, n$ and/or $ t $.

Stefan Hilger \cite{h} introduced the notion of time scale in 1988
in order to unify the theory of continuous and discrete calculus. The time scales approach not only unifies differential and difference equations, but also solves some other problems powerfully, such as a mix of stop-start and continuous behaviors \cite{js}, \cite{abl}. Nowadays the theory on time scales has been widely applied to several scientific fields such as biology, heat transfer, stock market, wound healing and epidemic models. 

Under the general form of problem $(P)$ it included the Emden-Fowler equation which arises, for example, in astrophysics, related to the stellar structure (gaseous dynamics). In this case the fundamental problem is to investigate the equilibrium configuration of the mass of spherical clouds of gas. Also arises in gas dynamics and fluid mechanics. The solutions of physical interest in this context are bounded non oscillatory and possess a positive zero.
In the relativistic mechanics and nuclear physics. And in chemically reacting systems: in the theory of diffusion and reaction this equation appears as governing the concentration $u$ of a substance which disappears by an isothermal reaction at each point of a slab of catalyst.

We refer to Wong \cite{w}, for a general historical overview about this equation.

Many works on this system have been written in the continuous case, we can cite among others, \cite{aoll}, \cite{hz} or \cite{z} for $n=1$ or \cite{zho} for $n=2$. 

On the discrete case we find the book \cite{abm} which studies the oscillation properties of the solutions of different difference equations. For the specific problem $u^{\Delta\Delta}(t)+p(t)u^\gamma(\sigma(t))=0$, where $p\geq 0$ and $\gamma$ quotient of odd positive numbers, also oscillation properties were studied in \cite{ah}.

On time scales some results on existence and uniqueness of classical solutions or solutions in the sense of distribution for $n=1$ can be found in the articles \cite{vivi}, \cite{kno}, \cite{zlg} and \cite{aopv}. Considering classical solutions, oscillation properties have also been studied, in works such as \cite{hss} (with delay) or \cite{bep}.

In the present paper we present some results on time scales considering classical solutions which generalize the ones from the continuous case.
 
The remainder of the paper is organized as follows. In Sect. $2$, we state some existence results supposing the existence of a pair of lower and upper solutions and employing the Schauder fixed point theorem. In Sect. $3$, we shall give a necessary and sufficient condition for the existence of positive solutions of the singular boundary value problem $(P)$  by constructing a lower solution.

\section{Lower and Upper Solutions Method}

Let $\T$ be an arbitrary time scale. We assume that $\mathrm{\T}$ has the topology that it inherits from the standard topology on $\R$. 
See \cite{bp} for general theory on time scales.

Let $a, b \in \T$, such that $a<\rho( b)$. If a is a right-dense point, we consider $J=(a,\sigma^2(b)]_\T$, $J^{\kappa}=(a,\sigma(b)]_\T$ and $(J^\kappa)^o=(a,\sigma(b))_\T$. In the other case,
$J=[a,\sigma^2(b)]_\T$, $J^{\kappa}=[a,\sigma(b)]_\T$ and $(J^\kappa)^o=[a,\sigma(b))_\T$.

The problem we will consider in this section is
$$
(P)\left\{\begin{array}{l}
x_i^{\Delta\Delta}=f_i(t,x^\sigma(t)),\; t\in (J^\kappa)^o,\\
\noalign{\bigskip}x_i(a)=A_i, \, x_i(\sigma^2(b))=B_i,\; i=1,2,\dots,n.
\end{array}\right.
$$
with $f=(f_1,\dots,f_n)$, $i=1,2,\dots,n$, where $f_i:(J^\kappa)^o\times\mathcal{A}\rightarrow \R$,  $\mathcal{A} \subset \R^n$.

We say that f verifies the hypothesis ${\rm (H_1)}$ if for every $i=1,2,\dots,n$ are satisfied the following conditions:
\begin{itemize}
\item [i)] For every $x\in \mathcal{A}$, $f_i(\cdot,x)\in C_{rd}((J^\kappa)^o)$,
\item [ii)] $f_i(t,\cdot)$ is continuous on $ \mathcal{A}$ uniformly in $t\in(J^\kappa)^o$.
\end{itemize}

For convenience, we denote
$$E\ =\ \left\{g\in C_{rd}((J^\kappa)^o,\R^+): \; \int_a^{\sigma(b)}(\sigma(s)-a)(\sigma^2(b)-s)g(s)\Delta s<+\infty\right\}.$$
We say that $f$ satisfies the condition ${\rm (H_2)}$ on
 $ \mathcal{B} \subset (J^\kappa)^o\times\mathcal{A}$
if for $i=1,2,\dots,n$ there exists a function $h_i\in E$ such that:
\begin{itemize}
\item [${\rm (H_2)}$] \hspace{4cm} $ |f_i(t,x)|\leq h_i(t),\hspace{0.25cm}\forall (t,x)\in \mathcal{B}$.\\
\end{itemize}
\begin{definicion}
A solution of $(P)$ is a function $x=(x_1,\dots,x_n)$, with $x_i\in
C^{2}_{rd}((a,b)_{\mathrm{\T}})$ for all $i=1,\dots,n$ such that  $x(t) \in \mathcal{A}$, for all $t \in [a,\sigma^{2}(b)]_{\mathrm{\T}} $, which satisfies $(P)$ for each $t\in (J^\kappa)^o$ and $i=1,\dots,n$, where
$$C^{2}_{rd}((a,b)_{\mathrm{\T}})=\left\{y\in C([a,\sigma^{2}(b)]_{\mathrm{\T}}), \ y^{\Delta\Delta}:(J^\kappa)^o \rightarrow \R, \mbox{ and }y^{\Delta\Delta}\in C_{rd}((a,b)_{\mathrm{\T}})\right\}.$$
\end{definicion}

\begin{definicion}
We say that $\alpha=(\alpha_1,\dots,\alpha_n)$, with $\alpha_i \in C^{2}_{rd}((a,b)_{\mathrm{\T}})$, is a lower solution of $(P)$ if $\alpha(t) \in \mathcal{A}$, for all $t \in [a,\sigma^{2}(b)]_{\mathrm{\T}} $ and 
$$
\begin{array}{lcl}
-\alpha^{\Delta\Delta}(t) \leq f(t,\alpha^{\sigma}(t)), ~ &t \in (J^\kappa)^o, \\
\\
\alpha(a)\leq A,\,  \alpha(\sigma^{2}(b))\leq B.
\end{array}
$$
\end{definicion}
An upper solution  $\beta=(\beta_1,\dots,\beta_n)$ of $(P)$ is defined similarly  by reversing the previous inequalities.

We have the following result

\begin{teor}\label{1.1}
Let  $\alpha$ and $\beta$ be, respectively, a lower and upper solution for problem $(P)$, such that $\alpha\leq\beta$ on  $[a,\sigma^2(b)]_\T$. If  $f$   satisfies     ${\rm (H_1)}$  and  the  conditions ${\rm (H_2)}$  on
 $$ \mathcal{D}_\alpha^\beta=\left\{(t,x)\in(J^\kappa)^o\times\R : \alpha^\sigma(t)\leq x\leq\beta^\sigma(t)\right\},$$
   and
 \begin{itemize}
 \item [${\rm (H_3)}$]    For $t \in [a,b]_T$ and $ x \in \mathcal{A}: \alpha(t)\leq x\leq\beta(t) \Rightarrow $ $$f_i(t,\alpha(t))\leq f_i(t,x)\leq f_i(t,\beta(t)),$$  for all $i=1,\dots,n$.
 \end{itemize}
   Then  problem $(P)$ has at least one solution $x$ such that $\alpha\leq x\leq\beta$ on $[a,\sigma^2(b)]_\T$.
\label{th_cl_1}\end{teor}
\begin{proof}
We consider the following modified problem
$$
(P_m)\left\{\begin{array}{lcl}
x_i^{\Delta\Delta}(t) = - f_i^*(t,x^{\sigma}(t)), ~ &t \in (J^\kappa)^o,
\\ \noalign{\bigskip} x_i(a)=A_i,\,  x_i(\sigma^{2}(b))=B_i, \; i=1,2,\dots,n.
\end{array}
\right.
$$
with
$$
f_i^*(t,x) = f_i(t,d_i(t,x))+\frac{d_i(t,x)-x_i}{1+|d_i(t,x)-x_i|},
$$
where $d=(d_1,\dots, d_n)$ and $d_i:[a,\sigma^2(b)]_\T\times\mathcal{A}\rightarrow \R$ is defined
$$
d_i(t,x) = \left\{
\begin{array}{lcl}
\alpha_i^{\sigma}(t), &\mbox{if} & x_i < \alpha_i^{\sigma}(t), \\
\\
  x_i, &\mbox {if}& \alpha_i^{\sigma}(t) \leq x_i \leq \beta_i^{\sigma}(t),\\
  \\
\beta_i^{\sigma}(t),  &\mbox{if}& \;  \beta_i^{\sigma}(t) \leq x_i.
\end{array}
\right.
$$
We can prove that $d_i(t,\cdot)$ is continuous on $\mathcal{A}$ uniformly in $t$ and $d_i(\cdot,x)\in C_{rd}([a,\sigma^2(b)]_\T$, for every $x\in \mathcal{A}$. Hence the function $p: \mathcal{D}_\alpha^\beta \rightarrow \mathcal{D}_\alpha^\beta$, $p(t,x)= (t,d(t,x))$ verified for each $x\in \mathcal{A}$, $p(\cdot,x)\in C_{rd}((J^\kappa)^o)$.

Due to the hypothesis it is easy to see that  ${\rm (H_1)}$ is satisfied and that there exist $h_i^*\in E$ such that ${\rm (H_2)}$ holds for the function $f^*$.

Note that, if $u$ is a solution of $(P_m)$ such that $\alpha\leq u\leq\beta$ on $[a,\sigma^2(b)]_\T$ then $u$ is a solution of $(P)$.

To show that any solution $u$ of $(P_m)$, is between $\alpha$ and $\beta$, suppose that there exists $i=1,\dots,n$ and $t^*\in[a,\sigma^2(b)]_\T$ such that $v_i(t^*)=\alpha_i(t)-u_i(t)>0$. As $v_i(a)\leq 0$ and $v_i(\sigma^2(b))\leq 0$, then there exists $t_0\in (a,\sigma^2(b))_\T$ with
$$
v_i(t_0)=\max\{v_i(t),\ t\in[a,\sigma^2(b)]_\T\}>0,
$$
and $v_i(t)<v_i(t_0)$ for $t \in (t_0,\sigma^2(b)]_\T$. The point
$t_0$ is not simultaneously left-dense and right-scattered (see theorem 2.1 in \cite{kno}) (this implies that $(\sigma\circ\rho)(t_0)=t_0$) and we have that $v_i^{\Delta\Delta}(\rho(t_0))\leq 0$ (see \cite{kno}), so given that $\alpha^\sigma(t)\leq d(t,u^\sigma(t))$
$$
\begin{array}{lcl}
-v_i^{\Delta\Delta}(\rho(t_0))&=& u_i^{\Delta\Delta}(\rho(t_0))-\alpha_i^{\Delta\Delta}(\rho(t_0))\leq -f_i(\rho(t_0),d(\rho(t_0),u^\sigma(\rho(t_0)))\\
\\
&-&\displaystyle \frac{d_i(\rho(t_0),u^\sigma(\rho(t_0)))-x_i^\sigma(\rho(t_0))}{1+|d_i(\rho(t_0),u^\sigma(\rho(t_0)))-x_i^\sigma(\rho(t_0))|}+f(\rho(t_0),\alpha^\sigma(\rho(t_0)))
\\
\\ &=& -\displaystyle\frac{\alpha(t_0) - u(t_0)}{1 + (\alpha(t_0)-u(t_0))}<0.\\
\end{array}
$$
So $v^{\Delta\Delta}(\rho(t_0))> 0$, that is a contradiction. And so we have proved that $v(t)\leq 0$, for each $ t\in[a,\sigma^2(b)]_\T$.

Analogously it can be proved that $u(t)\leq\beta(t)$, for all $t\in[a,\sigma^2(b)]_\T$.

We only need to prove that problem $(P_m)$ has at least one solution.

Consider now the operator $N:C([a,\sigma^2(b)]_\T)\rightarrow C([a,\sigma^2(b)]_\T)$, defined by
\begin{equation}\label{operadorN}
\displaystyle Nu(t)=\phi(t)+\displaystyle\int_a^{\sigma(b)}G(t,s)\ f^*(s,u^\sigma(s))\ \Delta s,
\end{equation}
for each $t \in [a,\sigma^2(b)]_\T$, where (see \cite{bp})
\begin{equation}
\label{green}
G(t,s)=\frac{1}{\sigma^2(b)-a}\left\{\begin{array}{lc}
(t-a)\ (\sigma^2(b)-\sigma(s)),&t\leq s,\\
\\
(\sigma(s)-a)\ (\sigma^2(b)-t),&\sigma(s)\leq t,
\end{array}\right.
\end{equation}
is the Green's function of the problem
$$
\left\{\begin{array}{l}
-x^{\Delta\Delta}=0,\\
\\
x(a)=x(\sigma^2(b))=0,
\end{array}\right.
$$
and, for $t\in [a,\sigma^2(b)]_\T$
$$\phi(t)=A+\displaystyle\frac{B-A}{\sigma^2(b)-a}\ (t-a),$$
 is the solution of $-x^{\Delta\Delta}=0$ such that $\phi(a)=A$ and $\phi(\sigma^2(b))=B$.

 Clearly, $G(t,s)>0$ on $(a,\sigma^2(b))_\T\times(a,\sigma^2(b))_\T$, $G(t,\cdot)$ is rd-continuous on $[a,\sigma(b)]_\T$ and $G(\cdot,s)$ is continuous on $[a,\sigma^2(b)]_\T$.

 The function $Nu$ defined by $(\ref{operadorN})$ belongs to $C([a,\sigma^2(b)]_\T)$
 because $f^*$ satisfies the conditions ${\rm (H_1)}$ and ${\rm (H_2)}$ on $(J^\kappa)^o \times \R $ and
 $G(t,s) \le s(1-s)$, for each $t, s \in [a,\sigma^2(b)]_\T$ .

It is obvious that $u\in C([a,\sigma^2(b)]_\T)$ is a solution of $(P_m)$ if and only if $u=Nu$. So the problem now is to ensure the existence of fixed points of $N$.

First of all, $N$ is well defined, is continuous and $N(C([a,\sigma^2(b)]_\T))$ is a bounded set. The existence of a fixed point of $N$ follows from the Schauder fixed point theorem, once we have checked that $N(C([a,\sigma^2(b)]_\T))$ is relatively compact, that using the Ascoli-Arzela theorem, is equivalent to proving that $N(C([a,\sigma^2(b)]_\T))$ is an equicontinuous family.

Let $h^*\in E$ be the function related to $f^*$ by condition ${\rm (H_2)}$. We compute the first derivative of $Nu$ using Theorem 1.117 of \cite{bp}\\
$$\begin{array}{lll}
|(Nu)^\Delta(t)|&=& \displaystyle\frac{1}{\sigma^2(b)-a}\ \bigg| \ (B-A) - \displaystyle\int_a^t(\sigma(s)-a)\ f^*(s,u^\sigma(s))\ \Delta s \\
\noalign{\bigskip}& &+\displaystyle\int_t^{\sigma(b)} (\sigma^2(b)-\sigma(s))\ f^*(s,u^\sigma(s))\ \Delta s \ \bigg|\leq\\
\noalign{\bigskip}&\leq&\displaystyle\frac{1}{\sigma^2(b)-a}\ \bigg(\ |B-A|+\displaystyle\int_a^t(\sigma(s)-a) \ h^*(s)\ \Delta s  \\
\noalign{\bigskip}& &  +\displaystyle\int_t^{\sigma(b)} (\sigma^2(b)-\sigma(s))\ h^*(s)\ \Delta s\ \bigg)\\
\noalign{\bigskip}&:=&\displaystyle\frac{1}{\sigma^2(b)-a}\bigg(\ |B-A|+\lambda(t)\ \bigg).\bigskip
\end{array}$$

Finally it is enough to check that $\lambda\in L^1((J^\kappa)^o)$, using integration by parts we obtain
$$\begin{array}{lll}
\displaystyle\int_a^{\sigma(b)}|\lambda(s)|\Delta s&=& \displaystyle\int_a^{\sigma(b)}\lambda(s)\Delta s =\lim_{r\to\sigma(b)^-}\displaystyle\int_a^r \bigg(\int_a^t\big(\sigma(s)-a\big)\ h^*(s)\ \Delta s \bigg)\Delta t\\
\noalign{\bigskip}& &+\displaystyle \lim_{r\to a^+}\displaystyle\int_r^{\sigma(b)}\bigg( \int_t^{\sigma(b)} \big(\sigma^2(b)-\sigma(s)\big)\ h^*(s)\ \Delta s \bigg)\Delta t\\
\noalign{\bigskip}&=&  2\ \displaystyle\int_a^{\sigma(b)}\big(\sigma(s)-a\big)\ \big(\sigma^2(b)-\sigma(s)\big)\ h^*(s)\ \Delta s\\
\noalign{\bigskip}& &-\displaystyle \lim_{r\to\sigma(b)^-}\big(\sigma^2(b)-r\big)\int_a^r\big(\sigma(s)-a\big)\ h^*(s)\ \Delta s \\
\noalign{\bigskip}& & -\displaystyle \lim_{r\to a^+}\big(r-a\big)\int_r^{\sigma(b)}  \big(\sigma^2(b)-\sigma(s)\big)\ h^*(s)\ \Delta s<+\infty,\\
\end{array}$$
due to $h^*\in E$, and the fact
$$\big(\sigma^2(b)-r\big)\int_a^r\big(\sigma(s)-a\big)\ h^*(s)\ \Delta s \le \displaystyle\int_a^{\sigma(b)}\big(\sigma(s)-a\big)\ \big(\sigma^2(b)-\sigma(s)\big)\ h^*(s)\ \Delta s,$$
$$\big(r-a\big)\int_r^{\sigma(b)}  \big(\sigma^2(b)-\sigma(s)\big)\ h^*(s)\ \Delta s \le \displaystyle\int_a^{\sigma(b)}\big(\sigma(s)-a\big)\ \big(\sigma^2(b)-\sigma(s)\big)\ h^*(s)\ \Delta s.$$
 And so the result is proved.
\qed   \end{proof}

\begin{nota}
The above theorem is true also if we change $(H_3)$ by:
\begin{itemize}
\item [${\rm (\overline{H}_3)}$]    For $t \in [a,b]_T$ and $ x, y \in \mathcal{A}$, there exists $M>0$ such that  $$f_i(t,x)- f_i(t,y)\leq M(x-y).$$
\end{itemize}
\end{nota}

\begin{nota}
The existence of lower solution and upper solution with $0<\alpha\le\beta$ can be obtained through conditions of $f_{i}$. For instance, if $f_{i}\in C_{rd}(J)$ and $f_{i}$ is bounded the existence holds.
\end{nota}

\section{ Existence of positive solution}

Consider the problem
$$
(P_0)\left\{\begin{array}{lcl}
-x_i^{\Delta \Delta}(t)= f_i(t,x^\sigma(t)),& & t\in {(J^\kappa)}^o,\\
\noalign{\bigskip} x_i(a)= x_i(\sigma^{2}(b))=0, \; i=0,\dots, n.
\end{array}
\right.
$$
We will deduce the existence  of solution  to $(P_0)$  by supposing
that the following hypothesis hold
\begin{itemize}
\item [${\rm (\widetilde{H}_1)}$] For every $i=1,\dots, n$,
$f_i:(J^\kappa)^o\times\mathcal{A}\rightarrow [0,+\infty)$,  where
$\mathcal{A}\subset(0,+\infty)\times \dots^n \times(0,+\infty)$, verifies
\begin{itemize}
\item [i)] For every $x\in \mathcal{A}$, $f_i(\cdot,x)\in C_{rd}((J^\kappa)^o)$,
\item [ii)] $f_i(t,\cdot)$ is continuous on $ \mathcal{A}$ uniformly in $t\in(J^\kappa)^o$.
\end{itemize}
  \item [${\rm (\widetilde{H}_2)}$] For every $i=1,\dots, n$, and $j=1,\dots, n$ there exists constants $\lambda_{ij}, \mu_{ij}$, with $-\infty<\lambda_{ij}<\mu_{ij}<1$, $\lambda_{ii}<0<\mu_{ii}<1$, $\mu_{ij}<0$ if $i\not=j$, such that if $0<c\le 1$ then $$c^{\mu_{ij}} f_i(t,x_1,\dots,x_n)\le f_i(t,x_1,\dots,cx_j,\dots,x_n)\le c^{\lambda_{ij}}f_i(t,x_1,\dots,x_n),$$
      for each $t\in (J^\kappa)^o$ and $x\in \mathcal{A}$
\end{itemize}

\begin{nota} If $c\geq 1$ for every $i=1,\dots, n$
$$c^{\lambda_{ij}} f_i(t,x_1,\dots,x_n)\le f_i(t,x_1,\dots,cx_j,\dots,x_n)\le c^{\mu_{ij}}f_i(t,x_1,\dots,x_n),$$
      for each $t\in (J^\kappa)^o$ and $x\in \mathcal{A}$.

\end{nota}
We consider a solutions to the problem
\begin{definicion}
A positive solution of type $1$ of $(P_0)$ is a function $x=(x_1,\dots,x_n)$, with $x_i\in
C^{2}_{rd}((a,b)_{\mathrm{\T}})$ for all $i=1,\dots,n$ such that  $x(t) \in \mathcal{A}$ and $x_i(t)>0$, for all $t \in [a,\sigma^{2}(b)]_{\mathrm{\T}} $, which satisfies the equalities on $(P_0)$ for each $t\in (J^\kappa)^o$ and $i=1,\dots,n$, and exist and are finite the limits
$$\displaystyle\lim_{t \to a^+}x_i^{\Delta}(t) \quad \mbox{ and } \quad  \lim_{t \to \sigma(b)^-}x_i^{\Delta}(t).$$
\end{definicion}

\begin{definicion}
	We say that $\alpha\in C_{rd}^{2}$ is a \emph{lower solution} of $(P_0)$ if for each $i=1,2,\dots,n$ we have,
\begin{displaymath}
\left\{\begin{array}{lcl}
-\alpha_i^{\Delta \Delta}(t)\le f_i(t,x^\sigma(t)),& & t\in {(J^\kappa)}^o,\\
\noalign{\bigskip} \alpha_i(a)= \alpha_i(\sigma^{2}(b))=0.
\end{array}
\right.
\end{displaymath}

	Similary, $\beta\in C_{rd}^{2}$ is called an \emph{upper solution} of $(P_0)$ if for each $i=1,2,\dots,n$,
	\begin{displaymath}
\left\{\begin{array}{lcl}
-\beta_i^{\Delta \Delta}(t)\ge f_i(t,x^\sigma(t)),& & t\in {(J^\kappa)}^o,\\
\noalign{\bigskip} \beta_i(a)= \beta_i(\sigma^{2}(b))=0.
\end{array}
\right.
	\end{displaymath}
\end{definicion}

\begin{lema}
\label{le_I1_I2}
Suppose that ${\rm (\widetilde{H}_1)}$ and ${\rm (\widetilde{H}_2)}$ hold. If $x$ is a positive solution of type $1$ of $(P_0)$ then for each $i=1,\dots,n$ there are constants $I_{i1}, I_{i2}$, $0<I_{i1}< I_{i2}$ such that
$$I_{i1}e(t)\leq x_i(t) \leq e(t) I_{i2},$$
where $e(t)=\displaystyle \frac{(t-a)(\sigma^2(b)-t)}{\sigma^2(b)-a}$.
\end{lema}
\begin{proof}
	Integrate the equations of $(P_0)$ in ${(J^\kappa)}^o$, for $i=1, \dots,n$
$$-\displaystyle\int_a^{\sigma(b)} f_i(s,x^{\sigma}(s)) \, \Delta s = \int_a^{\sigma(b)} x_i^{\Delta \Delta }(s) \, \Delta s = \lim_{t \to \sigma(b)^-}x_i^{\Delta}(t) - \lim_{t \to a^+}x_i^{\Delta}(t) < +\infty .$$

	From (\ref{green}), we have
$$x_i(t) = \displaystyle\int_a^{\sigma(b)} G(t,s)\, f_i(s,x^{\sigma}(s)) \, \Delta s \leq \frac{(t-a)\ (\sigma^2(b)-t)}{\sigma^2(b)-a}\, \int_a^{\sigma(b)} f_i(s,x^{\sigma}(s)) \, \Delta s.$$

Since
$$
\begin{array}{lcl}
& & \displaystyle \int_a^{t} \frac{(s-a)\ (\sigma^2(b)-\sigma(s))(t-a)(\sigma^2(b)-t)}{(\sigma^2(b)-a)^2}\, f_i(s,x^{\sigma}(s)) \, \Delta s\\
\\
& \leq&  \displaystyle \int_a^t (\sigma(s)-a)(\sigma^2(b)-t)\, f_i(s,x^{\sigma}(s)) \, \Delta s,
\end{array}
$$
and
$$
\begin{array}{lcl}
& & \displaystyle\int_t^{\sigma(b)} \frac{(s-a)\ (\sigma^2(b)-\sigma(s))(t-a)(\sigma^2(b)-t)}{(\sigma^2(b)-a)^2}\, f_i(s,x^{\sigma}(s)) \, \Delta s\\
\\
& \leq&  \displaystyle\int_t^{\sigma(b)} (t-a)(\sigma^2(b)-\sigma(s))\, f_i(s,x^{\sigma}(s)) \, \Delta s,
\end{array}
$$
we have
$$\frac{(t-a)\ (\sigma^2(b)-t)}{(\sigma^2(b)-a)^2}\int_a^{\sigma(b)} \frac{(s-a)\ (\sigma^2(b)-\sigma(s))}{(\sigma^2(b)-a)}\, f_i(s,x^{\sigma}(s)) \, \Delta s \leq  x_i(t).$$

Thus, if we consider
$$I_{i1}= \frac{1}{(\sigma^2(b)-a)} \int_a^{\sigma(b)} \frac{(s-a)\ (\sigma^2(b)-\sigma(s))}{(\sigma^2(b)-a)}\, f_i(s,x^{\sigma}(s)) \, \Delta s,$$
$$I_{i2}=  \int_a^{\sigma(b)}  f_i(s,x^{\sigma}(s)) \, \Delta s.$$

It verifies
$$I_{i1}e(t)\leq x_i(t) \leq e(t) I_{i2}.$$
\qed   \end{proof}

\begin{lema}
\label{exis on alpha beta}
 If $\alpha$ and $\beta$ are lower and upper solutions of de problem $(P_0)$ such that $0<\alpha(t)\le \beta(t)$ for $t\in (J^\kappa)^o$, and ${\rm (\widetilde{H}_1)}$ and $(H_3)$ or $(\overline{H_{3}})$ hold. Then problem $(P_0)$ has a solution $x$ such that
$$\alpha\le x \le \beta.$$
If in addition there exists a function $h(t)=(h_1(t),\dots,h_n(t))$  with $h_i \in L^1((J^\kappa)^o))$ such that $$ |f_i(t,x)|\leq h_i(t),\hspace{0.25cm}\forall \ t \in (J^\kappa)^o \mbox{ and } \alpha_i(t) \le x_i \le \beta_i(t), \; i=1,\dots,n,$$
then solution $x$ is a positive solution of type $1$.
\end{lema}

\begin{proof}
Let's consider ${\{a_k\}}_{k\ge
1}, \, {\{b_k\}}_{k\ge 1}\subset (J^\kappa)^o$ two sequences such that
${\{a_k\}}_{k\ge 1}\subset {(a,(a+\sigma(b))/2)}_\T$ is strictly
decreasing to $a$ if $a=\sigma (a)$, and $a_k=a$ for all $k\ge 1$ if
$a<\sigma (a)$, and ${\{b_k\}}_{k\ge 1}\subset {((a+\sigma(b))/2,\sigma(b))}_\T$ is
strictly increasing to $\sigma(b)$ if $\rho(\sigma(b))=\sigma(b)$, $b_k=\rho(\sigma(b))$ for all $k\ge
1$ if $\rho(\sigma(b))<\sigma(b)$.

We denote as $D_k:={[a_k, b_k]}_\T\subset(J^\kappa)^o$, $k\ge 1$, and let $\{r_{i1}^k\}$ $\{r_{i2}^k\}$ sequences so that
$$\alpha_i(a_k)\le r_{i1}^k \le \beta_i(a_k),$$
$$\alpha_i(b_k)\le r_{i2}^k \le \beta_i(k_k).$$
For each $x\in \mathcal{A}$ define $$f_{ki}^*(\cdot,x):D_k \rightarrow [0,\infty),$$
for all $k\in \N$, $k\ge 1$ and $i=1,\dots,n$, as
$$f_{ki}^*(t,x)= f_i(t,d(t,x))+ \frac{d_i(t,x)-x_i}{1+|d_i(t,x)-x_i|}.$$
Consider the problems
$$
(P_k)\left\{\begin{array}{lcl}
x_i^{\Delta \Delta}(t)= -f_{ki}^*(t,x^\sigma(t)),& & t\in D_k,\\
\noalign{\bigskip} x_i(a_k)= r_{i1}^k, \ x_i(\sigma^{2}(b_k))=r_{i2}^k, \; i=1,\dots, n.
\end{array}
\right.
$$
 Due to the hypothesis $f_i$, $i=1,\dots,n$, by theorem \ref{1.1} we can ensure that there exists solution $(x_{k1},\dots,x_{kn})$ with $x_{ki} \in C_{rd}[a_k,\sigma^2(b_k)]$ such that
 $$\alpha_i(t)\le x_{ki}(t)\le \beta_i(t),$$
 with $t \in [a_k,b_k]$.
	Since $[a_1,b_1]\subset[a_k,b_k]$ for $k\in \N$ there exists $t_k \in [a_1,b_1]$ such that
 $$|x_{ki}^\Delta (t_k)|\le \displaystyle \frac{|x_{ki}(b_1)-x_{ki}(a_1)|}{|b_1-a_1|} \le \frac{2(\beta_{i}(b_1)-\beta_{i}(a_1))}{|b_1-a_1|}.$$
 Thus, we can find a sequence $\{t_k\}$ which converges to $t_0 \in [a_k,b_k]$ for $k \in \N$, satisfying
 $$x_{ki} (t_k)\rightarrow x_{0i}\in [\alpha_i(t_0),\beta_i(t_0)],$$
and
 $$x_{ki}^\Delta (t_k)\rightarrow x_{0i}^\Delta,$$
 for $i=1,\dots,n$ when $k\to\infty$.

 We note that $x_{ki}$ is the solution of
 $$y_i^{\Delta \Delta}(t)= -f_{i}(t,y^\sigma(t)),$$
 with $y_i(t_k)=x_{ki}(t_k)$ and $y_i^\Delta(t_k)=x_{ki}^\Delta(t_k).$

	Hence, due to an adaptation of Theorem $3.2$ in \cite{od} and by existence theorems, we can find a solution of the problem:	
 $$x_i^{\Delta \Delta}(t)= -f_{i}(t,x^\sigma(t)), \; x_i(t_0)=x_{0i}, \; x_{i}^{\Delta }(t_0)=x_{0i}^\Delta,\; i=1,\dots,n.$$
 
 This solution is defined in a maximal interval $W$ and we can find at least one sequence $\{x_{k}(t)\}$ that converges uniformly to $x(t)$ in the compact subintervals of $W$.
 
On the other hand, $\bigcup_{k=1}^{\infty}[a_k,b_k]=(J^\kappa)^o$ and $\alpha_i(t)\le x_{ki}(t)\le \beta_i(t)$ for $t \in [a_k,b_k]$, then $x$ is defined in $(J^\kappa)^o$ and $\alpha(t)\le x(t)\le \beta(t)$ for all $ t \in (J^\kappa)^o$.
 From the conditions on $\alpha$ and $\beta$ on the boundary it follows that
 $$x(a)=x(\sigma^2(b))=0,$$
so that $x$ is a solution of the problem $(P_0)$.

 Suppose there exists a function $h(t)=(h_1(t),\dots,h_n(t))$  with $h_i \in L^1((J^\kappa)^o))$ such that $$ |f_i(t,x)|\leq h_i(t),\hspace{0.25cm}\forall \ t \in (J^\kappa)^o) \mbox{ and } \alpha_i(t) \le x_i \le \beta_i(t), i=1,\dots,n,$$
then, we can assume that $|x_i^{\Delta \Delta}(t)|\le h_i(t)$, $i=1,\dots,n$, which implies that $x_i^{\Delta \Delta}$ is absolutely integrable on $[a,\sigma(b)]_{\mathrm{\T}}$ and $x_i^{\Delta}\in C[a,\sigma(b)]$, $i=1,\dots,n$, so $x$ is a positive solution type $1$.
\qed   \end{proof}

\begin{teor}
\label{cond necesaria y suficiente}
Suppose that $(\widetilde{H}_1)$, $(\widetilde{H}_2)$ and $(H_{3})$ or $(\overline{{H}_3})$ hold. There exists a positive solution type $1$ if and only if the following conditions hold
$$0 < \int_a^{\sigma(b)}  f_i(s,E^{\sigma}(s)) \, \Delta s < \infty,$$
for all $i=1,\dots,n$, where $E(t)=(e(t),\dots,e(t))$.
\end{teor}
\begin{proof}

\underline{Necessity}

Suppose that there exists $x=(x_1,\dots,x_n)$ positive solutions type $1$ of $(P_0)$. By Lemma $\ref{le_I1_I2}$ there are constants $I_{i1}, I_{i2}$, $0<I_{i1}< I_{i2}$ for each $i=1,\dots,n$ such that
$$I_{i1}e(t)\leq x_i(t) \leq e(t) I_{i2}.$$
Let $K>0$, such that $KI_{i2}\le 1$, $\frac{1}{K} \ge 1$, $i=1,\dots,n$. By $(\widetilde{H}_2)$  and the above inequality, it follows that
$$
\begin{array}{lcl}
f_i\left(t,x^{\sigma}(t)\right) &\ge  & \displaystyle \Bigg(\frac{1}{K}\Bigg)^{\lambda_{ii}} f_i\left(t,x_1^{\sigma}(t),\dots,\frac{K x_i^{\sigma}(t)}{e^{\sigma}(t)} e^{\sigma}(t), \dots,x_n^{\sigma}(t)\right)\\
\\
&\ge&  \displaystyle K^{\mu_{ii}-\lambda_{ii}}\Bigg(\frac{x_i^{\sigma}(t)}{e^{\sigma}(t)}\Bigg)^{\mu_{ii}} f_i\left(t,x_1^{\sigma}(t),\dots,e^{\sigma}(t), \dots,x_n^{\sigma}(t)\right))
\\
&\ge& K^{\mu_{ij}-\lambda_{ii}}I_{i1}^{\mu_{ii}} f_i(t,x_1^{\sigma}(t),\dots,e^{\sigma}(t), \dots,x_n^{\sigma}(t))\\
\\
& \ge & \displaystyle K^{ \sum_{j=1}^n (\mu_{ij}-\lambda_{ij})}I_{i1}^{\mu_{ii}} \prod_{j=1,j\not=i}^n I_{j2}^{\mu_{ij}} f_i(t,E^{\sigma}(t))\\
\end{array}
$$
hence
$$
\begin{array}{rl}
 \displaystyle  \int_a^{\sigma(b)}  f_i(s,E^{\sigma}(s)) \, \Delta s & \le \\
 \le  & \displaystyle K^{  \sum_{j=1}^n (\lambda_{ij}-\mu_{ij})}I_{i1}^{-\mu_{ii}} \prod_{j=1,j\not=i}^n I_{j2}^{-\mu_{ij}} \int_a^{\sigma(b)} f_i(s,x^{\sigma}(s))\, \Delta s\\
\\
=& \displaystyle K^{  \sum_{j=1}^n (\lambda_{ij}-\mu_{ij})}I_{i1}^{-\mu_{ii}} \prod_{j=1,j\not=i}^n I_{j2}^{-\mu_{ij}} (- x_i^{\Delta}(\sigma(b))+ x_i^{\Delta}(a)) <\infty.\\
\end{array}
$$

\underline{Sufficiency}.

 Suppose that there exists a constant $C\ge1$ such that $CI_{i1}\ge 1$ and $I_{i2}\le C$.

We consider
$$\alpha_i(t)= k_{i1} y_i(t),$$
$$\beta_i(t)= k_{i2} y_i(t),$$
with
$$y_i(t)= \int_a^{\sigma(b)}  G(t,s)f_i(s,E^{\sigma}(s)) \, \Delta s,$$
where $G(t,s)$  is Green's function $\ref{green}$ and $k_{i1}$ and $k_{i2}$ be determined below.

Note that $y_i$ satisfies
\begin{itemize}
            \item $y_i, \ y_i^{\Delta} \in C([a,\sigma^2(b)]_\T)$
            \item $y_i^{\Delta \Delta} \in C_{rd}((a,\sigma^2(b))_\T)$
\end{itemize}

We have,
$$e(t) I_{i1}\le y_i(t) \le e(t) I_{i2},$$
where
$$ I_{i1}= \frac{1}{\sigma^2(b)-a}  \int_a^{\sigma(b)} \frac{(s-a)\ (\sigma^2(b)-\sigma(s))}{\sigma^2(b)-a}\, f_i(s,e^{\sigma}(s) )\, \Delta s ,$$
and
$$ I_{i2}=   \int_a^{\sigma(b)}  f_i(s,e^{\sigma}(s)) \, \Delta s ,$$
Consider now 
$$k_{i1}= \displaystyle \min\Big\{1,\Big( C^{\sum_{j=1}^n (\lambda_{ij}-\mu_{ij})} \prod_{j=1}^n I_{j2}^{\lambda_{ij}}\Big)^{\frac{1}{1-\mu_{ii}}}\Big\} ,$$
$$k_{i2}= \displaystyle \max\Big\{1,\Big( C^{\sum_{j=1}^n (\mu_{ij}-\lambda_{ij})} \prod_{j=1}^n I_{j1}^{\lambda_{ij}}\Big)^{\frac{1}{1-\mu_{ii}}}\Big\} .$$
Given that $e^{\sigma}(t) I_{i1}\le y_i^{\sigma}(t)$ and $CI_{i1}\ge 1$, then $e^{\sigma}(t)\le C y_i^{\sigma}(t)$.
Since
$$f_i(t,\alpha_1^{\sigma}(t), \dots, \alpha_n^{\sigma}(t))= f_i\Big(t,\frac{k_{11}}{C}\frac{C y_1^\sigma(t)}{e^{\sigma}(t)} e^{\sigma}(t), \dots, \frac{k_{n1}}{C}\frac{C y_n^\sigma(t)}{e^{\sigma}(t)} e^{\sigma}(t)\Big),$$
we have that
$$f_i(t,\alpha^{\sigma}(t))\ge  \prod_{j=1}^{n}\Bigg(\frac{k_{j1}}{C}\Bigg)^{\mu_{ij}} \Bigg(\frac{C y_j^\sigma(t)}{e^{\sigma}(t)}\Bigg)^{\lambda_{ij}} f_i(t, E^{\sigma}(t))$$

$$\ge \prod_{j=1}^{n} k_{j1}^{\mu_{ij}} C^{\lambda_{ij}-\mu{ij}} I_{j2}^{\lambda_{ij}} f_i(t, E^{\sigma}(t)),$$
and
$$\prod_{j=1}^{n} k_{j1}^{\mu_{ij}} C^{\lambda_{ij}-\mu{ij}} I_{j2}^{\lambda_{ij}} \ge k_{i1}^{1- \mu_{ii}} \prod_{j=1}^{n} k_{j1}^{\mu_{ij}}\ge k_{i1} \prod_{j=1, i\not=j}^{n} k_{j1}^{\mu_{ij}} \ge k_{i1},$$
which implies that
$$f_i(t,\alpha^{\sigma}(t))\ge   k_{i1} f_i(t, E^{\sigma}(t))= - \alpha_i^{\Delta\Delta}(t).$$

In a similar way
$$f_i(t,\beta^{\sigma}(t))\le  \prod_{j=1}^{n}(k_{j2}C)^{\mu_{ij}} \Bigg(\frac{ y_j^\sigma(t)}{C e^{\sigma}(t)}\Bigg)^{\lambda_{ij}} f_i(t, E^{\sigma}(t))$$

$$\le \prod_{j=1}^{n} k_{j2}^{\mu_{ij}} C^{\mu_{ij}-\lambda{ij}} I_{j1}^{\lambda_{ij}} f_i(t, E^{\sigma}(t)) \le   k_{i2} f_i(t, E^{\sigma}(t))= - \beta_i^{\Delta\Delta}(t). $$

Thus there is a lower solution $\alpha$ and an upper solution $\beta$ of the problem $(P_0)$ that satisfy $0<\alpha_i(t)\le\beta_i(t)$ for $t \in (J^\kappa)^o$, $i=1,\dots,n$, $\alpha_i(a)=\beta_i(a)=\alpha_i(\sigma^2(b))=\beta_i(\sigma^2(b))=0$. Applying the Lemma $\ref{exis on alpha beta}$, problem $(P_0)$ has a solution $x$ such that $\alpha\le x \le \beta$. Note that, for $t \in (J^\kappa)^o$ and $\alpha\le x \le \beta$,
$$0\le f_i(t,x_i^{\sigma}(t))=   f_i\Bigg(t,\Bigg(\frac{k_{11}}{C}\Bigg) \Bigg(\frac{Cx_1^{\sigma}(t)}{k_{11}e^{\sigma}(t)}\Bigg)e^{\sigma}(t), \dots, \Bigg(\frac{k_{n1}}{C}\Bigg) \Bigg(\frac{Cx_n^{\sigma}(t)}{k_{n1}e^{\sigma}(t)}\Bigg) e^{\sigma}(t)\Bigg) $$ $$\le \Bigg(\frac{k_{11}}{C}\Bigg)^{\lambda_{i1}} \Bigg(\frac{Cx_1^{\sigma}(t)}{k_{11}e^{\sigma}(t)}\Bigg)^{\mu_{i1}} \cdots \Bigg(\frac{k_{n1}}{C}\Bigg)^{\lambda_{in}} \Bigg(\frac{Cx_n^{\sigma}(t)}{k_{n1}e^{\sigma}(t)}\Bigg)^{\mu_{in}} f_i(t,E^{\sigma}(t))\le h_i(t),$$
with $h_i(t)= K_i f_i(t,E^{\sigma}(t))$ and
$$K_i=\prod_{j=1}^{n}\Bigg(\frac{k_{j1}}{C}\Bigg)^{\lambda_{ij}-\mu_{ij}}\prod_{j=1}^{n}\max_{k=1,\dots,n}\{(I_{k1})^{\mu_{ij}},(I_{k2})^{\mu_{ij}}\}.$$

Due to the hypothesis, we can then ensure that
$$\int_a^{\sigma(b)}h_i(s) \; \Delta s < \infty,$$
for $i=1,\dots,n$, which implies the existence of a positive solution type $1$ of the problem $(P_0)$ such that $0<\alpha\le x \le \beta$.

\qed   \end{proof}

\begin{teor}
\label{cond necesaria}
 If there exists a positive solution of the problem and $\sigma^2(b)>0$, then the following conditions hold
$$0 < \int_a^{\sigma(b)}  (\sigma(s)-a)(\sigma(b)-\sigma(s))f_i(s,[\sigma^2(b)]) \, \Delta s < \infty,$$
for all $i=1,\dots,n$, with $[\sigma^2(b)]=(\sigma^2(b),\dots,\sigma^2(b))$.
\end{teor}
\begin{proof}
Fixed $i=1,\dots,n$, let's consider $C_0>0$, $C_0^i$ two constants such that $\displaystyle \frac{C_0 x_j^{\sigma}(t)}{\sigma^2(b)}\le 1$, if $j=1,\dots,n$ and $j\not=i$, $\displaystyle \frac{1}{C_0}\ge 1$, $\displaystyle \frac{C_0^i x_i^{\sigma}(t)}{\sigma^2(b)}\ge 1$ and $C_0^i\ge1$.

We have,
$$f_i(t,x^{\sigma}(t))\ge \left(\prod_{j=1,j\not=i}^n \left(\frac{1}{C_0}\right)^{\lambda{ij}}\left(\frac{C_0 x_j^{\sigma}(t)}{\sigma^2(b)}\right)^{\mu{ij}}\right) f_i(t,\sigma^2(b),\dots,x_i^{\sigma}(t),\dots,\sigma^2(b))$$
$$\ge \left(\prod_{j=1,j\not=i}^n \left(\frac{1}{C_0}\right)^{\lambda{ij}}\left(\frac{C_0 x_j^{\sigma}(t)}{\sigma^2(b)}\right)^{\mu{ij}} \right) \left(\frac{1}{C_0^i}\right)^{\mu{ii}}\left(\frac{C_0^i x_i^{\sigma}(t)}{\sigma^2(b)}\right)^{\lambda{ii}} f_i(t,[\sigma^2(b)]).$$

Hence,
$$f_i(t,[\sigma^2(b)]) \le M f_i(t,x^{\sigma}(t)),$$
where
$$M=\sup_{t \in [a,\sigma(b)]_\T} \left(\prod_{j=1,j\not=i}^n C_0^{\lambda{ij}-\mu{ij}}\left(\frac{x_j^{\sigma}(t)}{\sigma^2(b)}\right)^{-\mu{ij}} \right) \left(C_0^i\right)^{\mu{ii}}\left(\frac{C_0^i x_i^{\sigma}(t)}{\sigma^2(b)}\right)^{-\lambda{ii}}.$$

Let $t_i\in (a,\sigma(b))_\T$,
$$\int_a^{t_i}\int_t^{t_i}f_i(s,[\sigma^2(b)]) \Delta s \Delta t\le -M \int_a^{t_i} \int_t^{t_i} x_i^{\Delta\Delta}(s) \Delta s\Delta t= $$ $$-M \int_a^{t_i} \left(x_i^{\Delta}(t_i)-x_i^{\Delta}(t)\right) \Delta t= -M \left(x_i^{\Delta}(t_i)(t_i-a)-x_i(t_i)+x_i(a)\right)<\infty.$$
	
	Integrating by parts, we have
$$\int_a^{t_i}\int_t^{t_i}f_i(s,[\sigma^2(b)]) \Delta s \Delta t= \int_a^{t_i} (\sigma(t)-a) f_i(t,x_{\sigma^2(b)}) \Delta t<\infty.$$

In a similar way
$$\int_{t_i}^{\sigma(b)}\int_{t_i}^t f_i(s,[\sigma^2(b)]) \Delta s \Delta t\le -M \int_{t_i}^{\sigma(b)} \int_{t_i}^t x_i^{\Delta\Delta}(s) \Delta s\Delta t= $$ $$-M \int_{t_i}^{\sigma(b)} \left(x_i^{\Delta}(t)-x_i^{\Delta}(t_i)\right) \Delta t= -M \left(x_i(\sigma(b))-x_i(t_i)- x_i^{\Delta}(t_i)(\sigma(b)-t_i)\right)<\infty,$$
and integrating by parts
$$\int_{t_i}^{\sigma(b)}\int_{t_i}^t f_i(s,[\sigma^2(b)]) \Delta s \Delta t= \int_{t_i}^{\sigma(b)} (\sigma(b)-\sigma(t)) f_i(t,[\sigma^2(b)]) \Delta t<\infty.$$

Then, we concluded that
$$0<\int_{a}^{\sigma(b)} (\sigma(t)-a)(\sigma(b)-\sigma(t)) f_i(t,[\sigma^2(b)])  \Delta t <\infty.$$
\qed   \end{proof}

\begin{teor}
\label{cond lower}
 If the following conditions hold
$$0 < \int_a^{\sigma(b)}  (\sigma(s)-a)(\sigma(b)-\sigma(s))f_i(s,[\sigma^2(b)]) \, \Delta s < \infty,$$
for all $i=1,\dots,n$ and $\sigma^2(b)>0$, then there exists a lower solution to problem $(P)$.
\end{teor}
\begin{proof}

Consider the function

$$g_i(t)= \int_{a}^{\sigma(b)} G(t,s)\frac{(\sigma(s)-a)(\sigma^2(b)-\sigma(s))}{(\sigma^2(b)-a)^2} f_i(s,[\sigma^2(b)])  \Delta s.$$
Let's see that
$$0\le g_i(t)<\infty,$$

\begin{displaymath}\begin{array}{lll}
& \displaystyle \int_{a}^{\sigma(b)} G(t,s)\left(\frac{(\sigma(s)-a)(\sigma^2(b)-\sigma(s))}{(\sigma^2(b)-a)^2}\right)^{\mu_{ii}} f_i(s,[\sigma^2(b)])  \Delta s\\
\noalign{\bigskip} = &\displaystyle \int_{a}^{t} \frac{(\sigma(s)-a)(\sigma^2(b)-t)}{\sigma^2(b)-a}\left(\frac{(\sigma(s)-a)(\sigma^2(b)-\sigma(s))}{(\sigma^2(b)-a)^2}\right)^{\mu_{ii}} f_i(s,[\sigma^2(b)])  \Delta s\\
\noalign{\bigskip} +& \displaystyle \int_{t}^{\sigma(b)} \frac{(t-a)(\sigma^2(b)-\sigma(s))}{\sigma^2(b)-a}\left(\frac{(\sigma(s)-a)(\sigma^2(b)-\sigma(s))}{(\sigma^2(b)-a)^2}\right)^{\mu_{ii}} f_i(s,[\sigma^2(b)])  \Delta s\\
\noalign{\bigskip} \le &  \displaystyle \int_{a}^{\sigma(b)} \frac{(\sigma(s)-a)(\sigma^2(b)-\sigma(s))}{\sigma^2(b)-a}\left(\frac{(\sigma(s)-a)(\sigma^2(b)-\sigma(s))}{(\sigma^2(b)-a)^2}\right)^{\mu_{ii}} f_i(s,[\sigma^2(b)])  \Delta \\
\noalign{\bigskip} \le& \displaystyle \frac{1}{\sigma^2(b)-a} \int_{a}^{\sigma(b)} (\sigma(s)-a)(\sigma^2(b)-\sigma(s)) f_i(s,[\sigma^2(b)])  \Delta s <\infty.\\
\end{array}
 \end{displaymath}

Furthermore, $g_i(a)=g_i(\sigma^2(b))=0$.

If we consider
$$L_{i1}= \frac{1}{\sigma^2(b)-a} \int_{a}^{\sigma(b)} \left( \frac{(\sigma(s)-a)(\sigma^2(b)-\sigma(s))^{1+\mu_{ii}}}{(\sigma^2(b)-a)^{2\mu_{ii}}} \right) f_i(s,[\sigma^2(b)])  \Delta s .$$

Then,
$$g_i(t)\le L_{i1}.$$

On the other hand,
\begin{displaymath}\begin{array}{lll}
& \displaystyle \frac{(t-a)(\sigma^2(b)-t)}{(\sigma^2(b)-a)^2}\ L_{i1} \ \le   \\
\noalign{\bigskip}\le  & \displaystyle \int_{a}^{t} \frac{(\sigma(s)-a)(\sigma^2(b)-t)}{\sigma^2(b)-a}   \left( \frac{(\sigma(s)-a)(\sigma^2(b)-\sigma(s))}{(\sigma^2(b)-a)^{2}} \right)^{\mu_{ii}} f_i(s,[\sigma^2(b)])  \Delta s \\
\noalign{\bigskip} +& \displaystyle  \int_{t}^{\sigma(b)} \frac{(t-a)(\sigma^2(b)-\sigma(s))}{\sigma^2(b)-a}  \left( \frac{(\sigma(s)-a)(\sigma^2(b)-\sigma(s))}{(\sigma^2(b)-a)^{2}} \right)^{\mu_{ii}}  f_i(s,[\sigma^2(b)])  \Delta s  \\
\noalign{\bigskip} = & g_i(t).\\
\end{array}
 \end{displaymath}

Let $\alpha_i(t)= k_{i1}g_i(t)$, where
$$k_{i1}=\min\left\{1, \prod_{j=1}^{n}L_{j1}^{\mu_{ij}}\left(\frac{1}{\sigma^2(b)}\right)^{\lambda_{ij}}C_2^{\mu{ij}-\lambda{ij}}\right\}^{\frac{1}{1-\mu_{ii}}},$$
with $C_2$ a constant such that $\sigma^2(b)C_2L_{i1}\le 1$ and $\displaystyle \frac{1}{C_2\sigma^2(b)}\ge 1$.

Thus, if we note that $0<\mu_{ii}<1$ and $\mu_{ij}<0$ if $i\not=j$, we obtained
$$\displaystyle k_{i1}L_{i1} \left[\frac{(t-a)(\sigma^2(b)-t)}{(\sigma^2(b)-a)^2}\right]^{\mu{ii}}\le \alpha_i(t)^{\mu_{ii}},$$
and $$\alpha_j(t)^{\mu_{ij}}\ge L_{j1}^{\mu_{ij}}.$$

Hence,
\begin{displaymath}\begin{array}{ll}
 & \displaystyle f_i(t,\alpha^{\sigma}(t)) \ge \displaystyle\prod_{j=1}^{n}\left(\frac{1}{C_2\sigma^2(b)}\right)^{\lambda_{ij}}(C_2\alpha_i^{\sigma}(t))^{\mu_{ij}}f_i(t,[\sigma^2(b)])\\
\noalign{\bigskip} = &\displaystyle \prod_{j=1}^{n}\frac{1}{\sigma^2(b)^{\lambda_{ij}}}C_2^{\mu_{ij}-\lambda_{ij}}(\alpha_j^{\sigma}(t))^{\mu_{ij}}f_i(t,[\sigma^2(b)])\\
\noalign{\bigskip} \ge& \displaystyle \prod_{j=1,j\not=i}^{n}\frac{1}{\sigma^2(b)^{\lambda_{ii}+\lambda_{ij}}}
C_2^{\mu_{ij}-\lambda_{ij}}C_2^{\mu_{ii}-\lambda_{ii}}  \ L_{j1}^{\mu_{ij}}\ (\alpha_i^{\sigma}(t))^{\mu_{ii}}  f_i(t,[\sigma^2(b)])\\
\noalign{\bigskip} \ge& \displaystyle k_{i1} \left(\frac{(\sigma^2(b)-\sigma(t))(\sigma(t)-a)}{(\sigma^2(b)-a)^2}\right)^{\mu_{ii}}
f_i(t,[\sigma^2(b)])= - \alpha_i^{\Delta\Delta}(t).\\
\end{array}
 \end{displaymath}

 This implies that $\alpha$ is a lower solution of the problem $(P_0)$.

\end{proof}

\begin{teor}
Suppose that the conditions of above theorem are satisfied, and consider $\alpha$ the lower solution of the problem $(P_0)$ provided. If there exists, $\beta$, an upper solution of $(P_0)$ with $0<\alpha\le\beta$ and $(\widetilde{H}_{1})$ and $(H_{3})$ or $(\overline{H}_{3})$ hold. Then there exists $x$ a positive solution of $(P_0)$.
\end{teor}
\begin{proof}
	The demonstration of this fact is immediate taking into account the construction of the lower solution $\alpha$ obtained in the previous theorem, the existence of the upper solution $\beta$ with $0<\alpha\le\beta$ and the implementation of the lemma \ref{exis on alpha beta}.

\end{proof}

\subsection{Particular cases}
Let us briefly consider the following examples
\begin{enumerate}
\item If $[a,\sigma^2(b)]_\T$ is bounded and consists of only isolated points, such as in the case $\T=h \Z$, then  the conditions on Theorems \ref{cond necesaria y suficiente}  and \ref{cond necesaria}  are fulfilled. This it follows from the fact
$$\int_a^{\sigma^2(b)}f(t) \ \Delta t= \sum_{t \in[a,\sigma^2(b))}(\sigma(t)-t)\ f(t).$$
\item Let $q>1$ fixed, the quantum time scale $\overline{q^{\Z}}$  is defined as
$$\overline{q^{\Z}}=\{q^k: k \in \Z \} \cup \{0\},$$
which appears throughout the mathematical physics literature, where the dynamical systems of interest are the q-difference equations.

Since the only non-isolated point is 0, the interesting case is the one in which the  interval containing it. We consider $a=0$ and $\sigma(b)=1$.

Taking into account the fact that 
$$\int_0^{1}f(t) \ \Delta t= \sum_{k=1}^{\infty}\frac{q-1}{q^k}\ f(q^{-k}),$$
we have
$$ \int_0^{1}  f_i(s,E^{\sigma}(s)) \, \Delta s=\sum_{k=1}^{\infty}\frac{q-1}{q^k}\ f_i \left(q^{-k},E^{\sigma}(q^{-k})\right),$$
with $E^{\sigma}(q^{-k})= \left( q^{-k+1}(1-q^{-k}), \dots, q^{-k+1}(1-q^{-k})\right)$. Hence, the convergence of this series is the necessary and sufficient condition in Theorem  \ref{cond necesaria y suficiente}.

Analogously, the condition in Theorem  \ref{cond necesaria} can be rewritten as
$$ 0 < \int_0^{1}  q^{-2k+1}(q-1)(1- q^{-k+1})f_i( q^{-k},(q,\dots,q) \, \Delta s < \infty.$$

\end{enumerate}

\vspace{0.5cm}

\noindent {\bf{ Competing interests:}} 

The authors declare that they have no competing interests.

\noindent {\bf{Authors’ contributions: }} 

All authors contributed equally in this article. They read and
approved the final manuscript.

\noindent {\bf{Acknowledgements.}} 

The authors are  grateful to
the referees for their valuable suggestions that led to the
improvement of  the original manuscript.\\
The research of V. Otero-Espinar has been partially supported by
 Ministerio de Educaci\'on y Ciencia (Spain) and FEDER,
Project MTM2010-15314.


\begin{thebibliography}{9}


\bibitem{h} Hilger, S.; \emph{Ein Maßkettenkalk¨ul mit Anwendung auf Zentrumsmannigfaltigkeiten}, (in German),
Universit¨at W¨urzburg, 1988.

\bibitem{js}  Jain, B. and Sheng, A. D., An exploration of the approximation of derivative functions via
finite differences, Rose-Hulman Undergraduate Math J. \textbf{8} (2007), 1--19.

\bibitem{abl}  Atici, F.M.,  Biles,  D.C. and  Lebedinsky, A., An application of time scales to economics, Math.
Comput. Modelling, \textbf{43} (2006), 718--726.

\bibitem{w} Wong, J. S. W., On the generalized Emden--Fowler equation, SIAM Rev., \textbf{17}, (1975), 339--360.

\bibitem{aoll} Agarwal, R. P.,  O'Regan, D.,  Lakshmikantham, V. and Leela, S., An upper and lower solution theory for singular Emden-Fowler equations, Nonlinear Anal. Real World Appl., \textbf{3}, (2002) 275--291.

\bibitem{hz}  Habets, P. and Zanolin,  F., Upper and Lower Solutions for a Generalized Emden--Fowler Equations,  J. Math. Anal. Appl. \textbf{181} (1994), 684--700.

\bibitem{z}  Zhang, Y.,  Positive Solutions of Singular Sublinear  Emden--Fowler Boundary Value Problems,  J. Math. Anal. Appl. \textbf{185} (1994), 215--222.

\bibitem{zho} Wei, Z., Positive solution of singular Dirichlet boundary value problems for second order differential equation system. J. Math. Anal. Appl. \textbf{328}  (2007),1255--1267.

\bibitem{abm} Agarwal, R. P.,  Bohner, M., Grace, S. R. and O'Regan, D.l, Discrete oscillation theory, Hindawi Publishing Corporation, New York, 2005

\bibitem{ah} Akin-Bohner, E. and Hoffacker, J., Oscillation properties of an {E}mden-{F}owler type equation on discrete time scales, J. Difference Equ. Appl., \textbf{9}, (2003) 603--612.

\bibitem{vivi} G\'omez, A. and Otero-Espinar, V.,  Existence and uniqueness of positive solution for singular BVPs on time scales, Adv. Difference Equ. 2009, Art. ID 728484, 12 pp.

\bibitem{kno} Khan, R. A,  Nieto, J. J. and Otero-Espinar, V.,  Existence and approximation of solution of three-point boundary value problems on time scale. J. Difference Equ. Appl. \textbf{14} (2008), 723--736.

\bibitem{zlg} Zhao, J., Lian, H. and Ge W., Existence of positive solutions for nonlinear m-point boundary value problems on time scales. Bound. Value Probl.  \textbf{2012:4}, (2012), 15 pp. 


\bibitem{aopv}  Agarwal R. P., Otero--Espinar V., Perera K. and Vivero D. R., Multiple Positive Solutions in the Sense of Distributions of Singular BVPs on Time Scales and a application to Emden--Fowler equations, Adv. Difference Equ. \textbf{2008} (2008), Article ID 796851, 13 pages.

\bibitem{hss} Han, Z., Sun, S. and Shi, B., Oscillation criteria for a class of second-order Emden-Fowler delay dynamic equations on time scales, J. Math. Anal. Appl., \textbf{334}, (2007) 847--858.

\bibitem{bep} Bohner, M., Erbe, L. and Peterson, A., Oscillation for nonlinear second order dynamic equations on a time scale, J. Math. Anal. Appl., \textbf{301}, (2005) 491--507.

\bibitem{bp}  Bohner, M. and A. Peterson, \emph{Dynamic Equations on Time Scales: An Introduction with Applications}, Birkhauser, Boston, 2001.
\bibitem{od} Hartman,P.; \emph{Ordinary Differential Equations}, Birkhauser Boston, 1982.





\end{thebibliography}
\end{document}